\documentclass[a4paper, reqno, 12pt]{amsart}
\usepackage[english]{babel}
\usepackage[T1]{fontenc}
\usepackage{mathptmx,amsmath,dsfont}
\usepackage{hyperref}
\usepackage{graphicx}
\usepackage{amssymb,verbatim,cases,bbold}
\usepackage{fullpage}
\usepackage[svgnames,hyperref]{xcolor}
\usepackage{color,soul}
\newtheorem{theorem}{Theorem}[section]
\newtheorem{definition}[theorem]{Definition}
\newtheorem{lemma}[theorem]{Lemma}
\newtheorem{proposition}[theorem]{Proposition}
\newtheorem{corollary}[theorem]{Corollary}
\newtheorem{e-definition}[theorem]{Definition}

\newtheorem{example}[theorem]{Example}

\def\n{\noindent}

\def\n{\noindent}

\begin{document}
\author{Thai Duong Do\textit{$^{1,2}$}, Pham Hoang Hiep\textit{$^{1}$}}
\address{\textit{$^{1}$}Institute for Artificial Intelligence, VNU University of Engineering and Technology, Hanoi, Vietnam.}
\address{\textit{$^{2}$}Department of Mathematics, National University of Singapore, 10 Lower Kent Ridge Road, Singapore, 119076, Singapore.}
\email{dtduong@vnu.edu.vn, dtduong@nus.edu.sg}
\address{\textit{$^{1}$}Institute for Artificial Intelligence, VNU University of Engineering and Technology, Hanoi, Vietnam.}
\email{phhiepvn@vnu.edu.vn}

\title{Comparison principles for Monge-Amp\`ere measures on pluripolar sets}

\date{\today}
\begin{abstract}
In this paper, we introduce a notion of singularity comparison for plurisubharmonic functions based on the Bedford–Taylor capacity. We establish comparison principles for the complex Monge–Amp\`ere operator on pluripolar sets in the Cegrell classes. As applications, we obtain a characterization of this relation via auxiliary functions in the energy class and prove a corresponding uniqueness result for the Monge–Amp\`ere equation. 
\end{abstract}
\maketitle

\tableofcontents
\section{Introduction and main results}

\n
We denote by $PSH(\Omega)$ the class of plurisubharmonic (psh) functions on $\Omega$, and by $PSH^-(\Omega)$ its subclass consisting of negative plurisubharmonic functions. Throughout this article, we always assume that $\Omega$ is a bounded hyperconvex domain. Recall that a domain $\Omega \subset \mathbb{C}^n$, $n \ge 1$, is called \emph{bounded hyperconvex} if it is bounded, connected, and open, and there exists a bounded plurisubharmonic function $\varphi : \Omega \to (-\infty,0)$ such that the sublevel set
\[
\{ z \in \Omega : \varphi(z) < c \}
\]
has compact closure in $\Omega$ for every $c \in (-\infty,0)$.

The Monge-Amp\`ere operator is defined on locally bounded plurisubharmonic functions in the sense of Bedford-Taylor (see \cite{BT76, BT82}). In\cite{De93}, Demailly showed that this operator can also be extended to certain classes of plurisubharmonic functions with isolated or compactly supported singularities.
In \cite{Ce98, Ce04}, Cegrell introduced a class $\mathcal{E}$ of plurisubharmonic functions on which the complex Monge-Amp\`ere operator $(dd^c \cdot)^n$ is well defined in the sense of pluripotential theory. The study of Monge-Amp\`ere measures associated with functions in $\mathcal{E}$, particularly on pluripolar sets, has since become an active area of research in pluripotential theory, with deep connections to singularity theory and complex Monge-Amp\`ere equations.

In this article, we introduce and study a notion of comparison of singularities measured in terms of the Bedford-Taylor capacity. We call this notion \emph{singularity comparison in capacity}, and we formalize it in the following definition.

\n
\begin{definition}\label{def:cap-order}
    Let $\varphi, \psi : \Omega \to [-\infty, +\infty]$ be functions defined on $\Omega$. We say that $\varphi$ is \emph{less singular in capacity} than $\psi$, and write $\varphi \preceq_{\mathrm{Cap}} \psi$ on $\Omega$, if
$$\varliminf_{t \to +\infty} t^n \, \mathrm{Cap}_\Omega\big(\{ z \in D : \varphi(z) < \psi(z) - t \}\big) = 0,$$
for every relatively compact subdomain $D \Subset \Omega$.
We say that $\varphi$ and $\psi$ have the same singularity type in capacity on $\Omega$, and write $\varphi \simeq_{\mathrm{Cap}}\psi$ if both $\varphi \preceq_{\mathrm{Cap}} \psi$ and $\psi \preceq_{\mathrm{Cap}} \varphi$ hold on $\Omega$.
\end{definition}

\n
The relation $\preceq_{\mathrm{Cap}}$ defines a preorder on the class of functions under consideration. Moreover, if $\varphi_1 \preceq_{\mathrm{Cap}} \psi_1$ and $\varphi_2 \preceq_{\mathrm{Cap}} \psi_2$ on $\Omega$, then for all $a_1, a_2 \ge 0$, we have
$$a_1 \varphi_1 + a_2 \varphi_2 \preceq_{\mathrm{Cap}} a_1 \psi_1 + a_2 \psi_2.$$

\n
We note that the singularity order $\preceq_{\mathrm{Cap}}$ introduced in Definition~1.1 is strictly weaker than the local singularity comparison considered in \cite{ACLR25} (see Proposition~\ref{prop:implication} and Example~\ref{strict}).

We begin by establishing a comparison principle for the complex Monge-Amp\`ere operator on pluripolar sets, which extends Lemma 4.1 in \cite{ACCH09}.

\n
\begin{theorem}\label{Comparion_Principle_Pluripolar_Set}
Let $\varphi_1, \ldots, \varphi_n, \psi_1, \ldots, \psi_n \in \mathcal{E}(\Omega)$ be such that $\varphi_i \preceq_{\mathrm{Cap}} \psi_i$ on $\Omega$ for all $1 \leq i \leq n$. Then
$$
\int_E dd^c \varphi_1 \wedge \cdots \wedge dd^c \varphi_n
\leq
\int_E dd^c \psi_1 \wedge \cdots \wedge dd^c \psi_n,
$$
for every pluripolar Borel set $E \subset \Omega$.
\end{theorem}

Next, we establish a comparison principle for Cegrell’s classes that strengthens Theorem 4.7 in \cite{KH09} and Theorem 3.1 in \cite{ACCH09}.

\begin{theorem}\label{Comparion_Principle1} 
Let $\varphi, \psi \in \mathcal{E}(\Omega)$ be such that either of the following two conditions holds:
\begin{itemize}
    \item[(i)] $\varliminf\limits_{z\to \xi } [ \varphi (z) - \psi (z) ] \geq 0$ for all $\xi\in\partial\Omega$;
    \item[(ii)] $\varphi \in \mathcal{N}(H)$ and $\psi \leq H$ on $\Omega$. 
\end{itemize}
Then for all $w_j \in PSH(\Omega) \cap L^\infty(\Omega)$, $-1 \leq w_j \leq 0$, $j = 1,2,\ldots,n$, we have the following inequality:
\begin{align*}
  \frac{1}{n!} &\int_{\{ \varphi < \psi \}} (\psi - \varphi)^n \ dd^c w_1 \wedge \cdots \wedge dd^c w_n 
+ \int_{\{\varphi = \psi = -\infty\}} (-w_1)(dd^c \max (\varphi,  \psi ) )^n\\
+ &\int_{\{\varphi < \psi\}} (-w_1)(dd^c \psi)^n
\leq \int_{\{\varphi < \psi\}} (-w_1)(dd^c \varphi)^n + \int_{\{\varphi = \psi = -\infty\}} (-w_1)(dd^c \varphi)^n.  
\end{align*}

\end{theorem}

\n
Combining Theorems \ref{Comparion_Principle_Pluripolar_Set} and \ref{Comparion_Principle1}, we obtain comparison principles for Cegrell’s classes that extend Corollary 3.2 in \cite{ACCH09} and Theorem 3.5 in \cite{ACLR25}.

\begin{theorem}\label{Comparion_Principle2} 
Let $\varphi, \psi \in \mathcal{E}(\Omega)$ be such that $\varphi\preceq_{\mathrm{Cap}}\psi$, and suppose that either of the following two conditions holds:
\begin{itemize}
    \item[(i)] $\varliminf\limits_{z\to \xi } [ \varphi (z) - \psi (z) ] \geq 0$ for all $\xi\in\partial\Omega$; 
    \item[(ii)] $\varphi \in \mathcal{N}(H)$ and $\psi \leq H$ on $\Omega$.  
\end{itemize}
Then for all $w_j \in PSH(\Omega) \cap L^\infty(\Omega)$, $-1 \leq w_j \leq 0$, $j = 1,2,\ldots,n$ we have the following inequality:

$$\frac{1}{n!} \int_{\{ \varphi < \psi \}} (\psi - \varphi)^n \, dd^c w_1 \wedge \cdots \wedge dd^c w_n + \int_{\{\varphi < \psi\}} (-w_1)(dd^c \psi)^n$$
$$\leq \int_{\{\varphi < \psi\}} (-w_1)(dd^c \varphi)^n .$$
\end{theorem}

\begin{theorem}\label{Comparion_Principle} 
Let $\varphi, \psi \in \mathcal{E}(\Omega)$ be such that $\varphi\preceq_{\mathrm{Cap}}\psi$, $(dd^c\varphi )^n \leq (dd^c\psi )^n$ on $\{ \varphi <\psi \}$ and suppose that either of the following two conditions holds:
\begin{itemize}
    \item[(i)] $\varliminf\limits_{z\to \xi } [ \varphi (z) - \psi (z) ] \geq 0$ for all $\xi\in\partial\Omega$;
    \item[(ii)] $\varphi \in \mathcal{N}(H)$, $\psi \leq H$ on $\Omega$.
\end{itemize}
Then $\varphi\geq\psi$.
\end{theorem}

\begin{theorem}\label{Comparion_Principle_Equal_Corollary} 
Let $\varphi, \psi \in \mathcal{E}(\Omega)$ be such that 
$$(dd^c \varphi)^n = (dd^c \psi)^n.$$
Assume that for every $z \in \Omega$, there exists a neighborhood $U$ of $z$ such that either $\varphi \preceq_{\mathrm{Cap}} \psi$ or $\psi \preceq_{\mathrm{Cap}} \varphi$ on $U$. Suppose, in addition, that one of the following conditions holds:
\begin{enumerate}
    \item[(i)] $\displaystyle \liminf_{z \to \xi} \bigl[\varphi(z) - \psi(z)\bigr] = 0$ for all $\xi \in \partial \Omega$;
    \item[(ii)] $\varphi, \psi \in \mathcal{N}(H)$.
\end{enumerate}
Then $\varphi = \psi$.
\end{theorem}

Finally, we characterize the singularity comparison in capacity.
\begin{theorem}\label{Characterization}
Let $\varphi \in \mathcal{E}(\Omega)$ and $\psi \in PSH^{-}(\Omega)$. Then the following statements are equivalent:
\begin{enumerate}
    \item[(i)] There exists a function $u \in \mathcal{E}^a(\Omega)$ such that
 $$\varphi \geq \psi + u;$$
    \item[(ii)]  $\varphi \preceq_{\mathrm{Cap}} \psi$.
\end{enumerate}
\end{theorem}
\n
As a direct consequence, we obtain the following.
\begin{corollary}
    Let $\psi \in PSH^{-}(\Omega)$. Then $\psi + u\simeq_{\mathrm{Cap}} \psi$ for every $u \in \mathcal{E}^a(\Omega).$
\end{corollary}
The paper is organized as follows. In Section 2, we recall the necessary definitions and background material on the Cegrell classes and the Bedford-Taylor capacity. We also discuss the relationship between the capacity-based singularity order introduced in Definition~\ref{def:cap-order} and the local singularity comparison of \cite{ACLR25}, showing in particular that the former is strictly weaker. Section 3 is devoted to the proofs of the main results: we establish the key technical lemmas (Lemmas 3.1--3.4) and then prove Theorems 1.2--1.7.
\section{Preliminaries and Comparison of Singularity Notions}
First, we recall the definitions of the Cegrell classes and the Bedford–Taylor Monge–Amp\`ere capacity. For further details and related results, we refer the reader to \cite{A07,ACCH09,Be06,B06,Ce98,Ce04,Ce08,DD20}.
\begin{definition}
    We say that a bounded plurisubharmonic function $\varphi$ on $\Omega$ belongs to the class $\mathcal E_0 (\Omega )$ if
\[
\lim_{z \to \xi} \varphi(z) = 0 \quad \text{for every } \xi \in \partial \Omega,
\]
and
\[
\int_{\Omega} (dd^c \varphi)^n < +\infty.
\]
\end{definition}
\begin{definition}
    Let $\mathcal F(\Omega)$ denote the class of plurisubharmonic functions $\varphi$ on $\Omega$ such that, there exists a decreasing sequence $\{ \varphi_j \}_{j=1}^\infty \subset\mathcal E_0 (\Omega )$ satisfying $\varphi_j \searrow \varphi$ pointwise on $\Omega$ as $j \to \infty$, and
\[
\sup_j \int_{\Omega} (dd^c \varphi_j)^n < +\infty.
\]
\end{definition}
\begin{definition}
  Let $\mathcal E(\Omega)$ denote the class of plurisubharmonic functions $\varphi$ on $\Omega$ such that, for every $z_0 \in \Omega$, there exists a neighborhood $\omega \subset \Omega$ of $z_0$ and a decreasing sequence $\{ \varphi_j \}_{j=1}^\infty \subset\mathcal E_0 (\Omega )$ satisfying $\varphi_j \searrow \varphi$ pointwise on $\omega$ as $j \to \infty$, and
\[
\sup_j \int_{\Omega} (dd^c \varphi_j)^n < +\infty.
\]  
\end{definition}

\n
A fundamental sequence $[\Omega_j]$ is an increasing sequence of strictly pseudoconvex domains in $\Omega$ such that
\[
\Omega_j \Subset \Omega_{j+1}, \qquad \bigcup_{j=1}^\infty \Omega_j = \Omega.
\]

\begin{definition}
 Let $\mathcal{K} \subset\mathcal E (\Omega )$. We define
$$\mathcal{K}^a = \left\{ \varphi \in \mathcal{K} : \int_E (dd^c \varphi)^n = 0 \ \text{for every pluripolar Borel set } E \subset \Omega \right\}.$$   
\end{definition}
\begin{definition}
    Let $u \in PSH(\Omega)$ with $u \le 0$, and let $[\Omega_j]$ be a fundamental sequence. Define
\[
u_j = \sup \{ \varphi \in \mathrm{PSH}(\Omega) : \varphi \le u \ \text{on } \Omega \setminus \overline { \Omega_j }\}.
\]
\end{definition}

\n
Then $u_j \in PSH(\Omega)$ and $u_j = u$ on $\Omega \setminus \overline{ \Omega_j }$. Moreover, $\{u_j\}$ is an increasing sequence, and hence $\lim_{j \to \infty} u_j$ exists quasi-everywhere on $\Omega$. We define
\[
\tilde{u} = \left( \lim_{j \to \infty} u_j \right)^*,
\]
which is plurisubharmonic on $\Omega$.

\n
Set
\[
\mathcal N (\Omega ) = \{ u \in \mathcal E (\Omega ) : \tilde{u} = 0 \}.
\]
Then $\mathcal N (\Omega )$ is a convex cone, and it consists precisely of those functions in $\mathcal E(\Omega )$ whose smallest maximal plurisubharmonic majorant is identically zero.

\begin{definition}
   Let $\mathcal K \in \{\mathcal E_0, \mathcal F, \mathcal N\}$. A plurisubharmonic function $u$ on $\Omega$ belongs to the class $\mathcal K(\Omega,H)$, where $H \in \mathcal E$, if there exists $\varphi \in \mathcal K$ such that
\[
H \ge u \ge H + \varphi.
\] 
\end{definition}
\begin{definition}
   The Bedford-Taylor Monge-Amp\`ere capacity (see \cite{BT82}) on $\Omega$ is defined for every Borel set $E \subset \Omega$ by
\[
\mathrm{Cap}_{\Omega} (E) = \sup \left\{ \int_E (dd^c u)^n : u \in \mathrm{PSH}(\Omega),\ -1 \le u \le 0 \right\}.
\] 
\end{definition}

Next, we recall the notion of singularity comparison introduced in \cite{ACLR25}, which we will refer to as the \textit{local singularity comparison} in order to distinguish it from our notion.
 
\begin{definition}[{\cite[Definition~3.1]{ACLR25}}]
Let $\varphi, \psi \in \mathcal{E}(\Omega)$. We say that $\varphi$ is
\emph{more singular than} $\psi$ on $\Omega$, written $\varphi \preceq \psi$,
if for every compact set $K \Subset \Omega$ there exists a constant
$C_{K}$ such that
\[
  \varphi \leq \psi + C_{K} \quad \text{on } K.
\]
We say that $u$ and $v$ have the \emph{same singularities},
written $u \simeq v$, if both $\varphi \preceq \psi$
and $\psi \preceq \varphi$.
\end{definition}
 
We prove that the singularity comparison in capacity is weaker than the local singularity comparison considered in \cite{ACLR25}. Roughly speaking, if $\varphi$ is more singular than $\psi$, then $\psi$ is less singular than $\varphi$ in capacity on $\Omega$.
\begin{proposition}\label{prop:implication}
Let $\varphi, \psi \in \mathcal{E}(\Omega)$. Then
\[
  \varphi \preceq \psi
  \implies
  \psi \preceq_{\mathrm{Cap}} \varphi.
\]
In particular, if $\phi \simeq \psi$ then $\phi \simeq_{\mathrm{cap}} \psi$.
\end{proposition}
\begin{proof}
Suppose $\varphi \preceq \psi$, and let
$K \Subset \Omega$ be a relatively compact subdomain.
By assumption there exists $C_{K} > 0$ such that
$\varphi \leq \psi + C_{K}$ on $K$.
Hence $\{\psi < \varphi - t\} \cap K = \emptyset$ for all $t > C_{K}$,
which gives
\[
  \mathrm{Cap}_{\Omega}\bigl(\{\psi < \varphi - t\} \cap K\bigr) = 0
  \quad \text{for all } t > C_{K}.
\]
In particular,
\[
  \varliminf_{t \to +\infty}
  t^{n}\,\mathrm{Cap}_{\Omega}\bigl(\{\psi < \varphi - t\} \cap K\bigr) = 0,
\]
so $\psi \preceq_{\mathrm{Cap}} \varphi$ by Definition~\ref{def:cap-order}.
\end{proof}
 We conclude this section with an example illustrating that the singularity comparison in capacity is strictly weaker than the local singularity comparison.
\begin{example}\label{strict}
Let $\mathbb{B} \subset \mathbb{C}^{n}$ be the unit ball with $n \geq 2$, $0<a<1$,
and set
\[
  \psi(z) = -(-\log|z|)^a, \qquad \varphi(z) = 0.
\]
Then $\varphi,\psi\in\mathcal{E}(\mathbb{B})$ and $\psi \preceq_{\mathrm{Cap}} \varphi$ but $\varphi \not\preceq \psi$. In particular, $\varphi \simeq_{\mathrm{Cap}} \psi$ but $\varphi \not\simeq \psi$.
\end{example}

\begin{proof}
We first observe that
\[
\lim_{t \to +\infty}
t^{n}\,\mathrm{Cap}_{\mathbb{B}}\bigl(\{\psi < - t\} \bigr)
= \lim_{t \to +\infty}\frac{t^n}{t^{n/a}} = 0.
\]
Hence, it follows from Proposition 3.4 in \cite{CKZ05} that $\psi\in\mathcal{F}^a(\mathbb{B})\subset\mathcal{E}(\mathbb{B})$.

We have $\varphi \preceq\psi$ and $\varphi\preceq_{\mathrm{Cap}}\psi$ since $\psi\leq\varphi$ on $\mathbb{B}$.

To show that $\psi \preceq_{\mathrm{Cap}} \varphi$, let $K\Subset \mathbb{B}$ be an arbitrary compact set. Then
\[
\varliminf_{t \to +\infty}
t^{n}\,\mathrm{Cap}_{\mathbb{B}}\bigl(\{\psi < \varphi - t\} \cap K\bigr)
\leq \varliminf_{t \to +\infty}
t^{n}\,\mathrm{Cap}_{\mathbb{B}}\bigl(\{\psi < \varphi - t\}\bigr)
= \varliminf_{t \to +\infty}\frac{t^n}{t^{n/a}} = 0.
\]
Therefore, $\psi \preceq_{\mathrm{Cap}} \varphi$.

To show that $\varphi \not\preceq \psi$, let $K \Subset \mathbb{B}$ be a compact set with $0 \in K$. If $\varphi \preceq \psi$, then there exists a constant $C_K$ such that
\[
-1 \leq -(-\log|z|)^a + C_{K},
\]
which is impossible as $|z| \to 0$. Hence, $\varphi \not\preceq \psi$.
\end{proof}
\section{Proof of main results}

First, we need the following lemma.
\begin{lemma}\label{Lemma1}
Let $\varphi \in \mathcal{F}(\Omega)$ and $\psi \in PSH(\Omega)$ satisfy
$$
\sup_{\Omega} \psi < 0.
$$
Then the following hold:
\begin{itemize}
    \item[(i)] $$
\int_{\Omega} (dd^c \varphi)^n = \int_{\Omega} (dd^c \max(\varphi, \psi))^n.
$$
\item[(ii)] For every Borel set $E \subset \Omega$ such that $\{\varphi \leq \psi\} \subset E$,
$$
\int_E (dd^c \varphi)^n = \int_E (dd^c \max(\varphi, \psi))^n.
$$
\end{itemize}
\end{lemma}

\begin{proof}
(i) By Theorem 2.1 in \cite{Ce04}, there exists a decreasing sequence $\{\varphi_j\} \subset \mathcal{E}_0(\Omega)$ such that $\varphi_j \searrow \varphi$ on $\Omega$. We have $\varphi_j = \max (\varphi_j, \psi )$ on $\Omega \setminus \{ \varphi_j < \psi \}$ and $\{ \varphi_j < \psi \}\Subset\Omega$.
It thus follows from Stokes’ theorem that
$$
\int_{\Omega} (dd^c \varphi_j)^n = \int_{\Omega} (dd^c \max(\varphi_j,\psi))^n.
$$
Passing to the limit as $j \to \infty$ yields
$$
\int_{\Omega} (dd^c \varphi)^n = \int_{\Omega} (dd^c \max(\varphi,\psi))^n.
$$

(ii) By Theorem 4.1 in \cite{KH09},
$$
(dd^c \varphi)^n = (dd^c \max(\varphi,\psi))^n \quad \text{on } \{\varphi > \psi\}.
$$
Hence,
$$
(dd^c \varphi)^n - (dd^c \max(\varphi,\psi))^n
$$
is supported in $\{\varphi \leq \psi\}$. Using (i), we obtain
$$
\int_{\{\varphi \leq \psi\}} (dd^c \varphi)^n = \int_{\{\varphi \leq \psi\}} (dd^c \max(\varphi,\psi))^n.
$$
The conclusion follows for every Borel set $E \supset \{\varphi \leq \psi\}$.
\end{proof}

Second, we need the following lemma, which is independently interesting.
\begin{lemma}\label{Lemma2}
Let $\varphi, \psi \in \mathcal{F}(\Omega)$. Then
$$(1- \epsilon )^n \int_{ \{ \varphi <- \frac { 2t } { \epsilon } \} } (dd^c \varphi )^n \leq \int_{ \{ \psi <  - t \} } (dd^c \psi )^n +  \Big( \frac { 2 } { \epsilon } - 1 \Big)^n t^n  \mathrm{Cap}_\Omega\Big(\{ \varphi \leq \psi - t\}\Big) ,$$
for all $0<\epsilon\leq 1$, $t\geq 0$.
\end{lemma}

\begin{proof}
The inequality is trivial if $\varphi \equiv 0$. Thus, we may assume that $\varphi \not\equiv 0$, which implies that $\varphi < 0$ on $\Omega$.
We set
$$u_t = \max \Big ( ( 1 - \epsilon ) \varphi - t, \psi , - \frac { 2t } { \epsilon } \Big ),$$
$$v_t = \max \Big ( ( 1 - \epsilon ) \varphi - t , - \frac { 2t } { \epsilon } \Big ) =\max \Bigg( ( 1 - \epsilon ) \varphi  , - t \Big( \frac { 2 } { \epsilon } - 1 \Big) \Bigg) - t.$$
We set
$$E_t=\Big\{ \varphi <- \frac { 2t } { \epsilon } \Big\}\cap \{ \varphi > \psi - t \}.$$
Since $(1-\epsilon)\varphi - t > \psi$ on $E_t$, we have $u_t = v_t$ on $E_t$. By Corollary 4.3 in \cite{BT87},
$$
\int_{E_t} (dd^c u_t)^n = \int_{E_t} (dd^c v_t)^n.
$$

Moreover, since
$$
E_t \cup \left\{\psi \leq \max\!\left((1-\epsilon)\varphi - t,\,-\frac{2t}{\epsilon}\right)\right\}
\subset \{\psi < -t\},
$$
Lemma \ref{Lemma1} implies
$$
\int_{E_t} (dd^c u_t)^n
\leq \int_{\{\psi < -t\}} (dd^c u_t)^n
= \int_{\{\psi < -t\}} (dd^c \psi)^n.
$$

On the other hand, by Lemma \ref{Lemma1},
\begin{align*}
\int_{E_t} (dd^c v_t)^n
&= \int_{\{\varphi < -\frac{2t}{\epsilon}\}} (dd^c v_t)^n
- \int_{\{\varphi < -\frac{2t}{\epsilon}\}\cap \{\varphi \leq \psi - t\}} (dd^c v_t)^n \\
&= \int_{\{\varphi < -\frac{2t}{\epsilon}\}}
\bigg(dd^c \max\Big((1-\epsilon)\varphi - t,\,-t\big(\tfrac{2}{\epsilon}-1\big)\Big)\bigg)^n \\
& - \int_{\{\varphi < -\frac{2t}{\epsilon}\}\cap \{\varphi \leq \psi - t\}}
\bigg(dd^c \max\Big((1-\epsilon)\varphi - t,\,-t\big(\tfrac{2}{\epsilon}-1\big)\Big)\bigg)^n \\
&\geq (1-\epsilon)^n
\int_{\{\varphi < -\frac{2t}{\epsilon}\}} (dd^c \varphi)^n  - \left(\tfrac{2}{\epsilon}-1\right)^n t^n \,
\mathrm{Cap}_\Omega(\{\varphi \leq \psi - t\}).
\end{align*}

Combining the above estimates yields
$$
(1-\epsilon)^n \int_{\{\varphi < -\frac{2t}{\epsilon}\}} (dd^c \varphi)^n
\leq \int_{\{\psi < -t\}} (dd^c \psi)^n
+ \left(\tfrac{2}{\epsilon}-1\right)^n t^n \,
\mathrm{Cap}_\Omega(\{\varphi \leq \psi - t\}).
$$
 
\end{proof}

\begin{lemma}\label{Lemma3}
Let $\varphi, \psi \in \mathcal{F}(\Omega)$ satisfy $\varphi \preceq_{\mathrm{Cap}} \psi$ on $\Omega$. 
%$$
%\varliminf\limits_{t \to +\infty} t^n \,\mathrm{Cap}_\Omega\big(\{z\in \Omega : \varphi (z) < \psi (z) - t\}\big) = 0,$$
Then
$$\int_{ \{ \varphi = -\infty \} } (dd^c \varphi )^n\leq \int_{ \{ \psi = -\infty \} } (dd^c \psi )^n.$$
\end{lemma}

\begin{proof}
By Lemma \ref{Lemma2}, we have
$$
(1-\epsilon)^n \int_{\{\varphi < -\frac{2t}{\epsilon}\}} (dd^c \varphi)^n
\leq \int_{\{\psi < -t\}} (dd^c \psi)^n
+ \left(\frac{2}{\epsilon}-1\right)^n t^n \,
\mathrm{Cap}_\Omega(\{\varphi \leq \psi - t\}).
$$

Since
$$
\liminf_{t \to +\infty} \,
t^n \mathrm{Cap}_\Omega(\{\varphi \leq \psi - t\}) = 0,
$$
it follows that
$$
(1-\epsilon)^n \int_{\{\varphi = -\infty\}} (dd^c \varphi)^n
\leq \int_{\{\psi = -\infty\}} (dd^c \psi)^n.
$$

Letting $\epsilon \to 0$ yields
$$
\int_{\{\varphi = -\infty\}} (dd^c \varphi)^n
\leq \int_{\{\psi = -\infty\}} (dd^c \psi)^n.
$$
\end{proof}

\begin{lemma}\label{Lemma4}
Let $\varphi, \psi \in \mathcal{E}(\Omega)$ satisfy $\varphi \preceq_{\mathrm{Cap}} \psi$ on $\Omega$. Then, for every pluripolar Borel set $E \subset \Omega$,
$$
\int_E (dd^c \varphi)^n \leq \int_E (dd^c \psi)^n.
$$
\end{lemma}
\begin{proof}
Without loss of generality, it suffices to prove the inequality when $E$ is a compact pluripolar set. By Lemma 4.3 in \cite{ACCH09}, there exists $\varphi_E \in \mathcal{F}(\Omega)$ such that $\varphi_E \geq \varphi$ and
$$
(dd^c \varphi_E)^n = 1_E (dd^c \varphi)^n,
$$
where $1_E$ denotes the characteristic function of $E$.

Fix a domain $D$ with $E \Subset D \Subset \Omega$, and define
$$
u = \sup \{ v \in PSH^-(\Omega) : v \leq \varphi_E \ \text{on } \Omega \setminus \overline{D} \}.
$$
Then $u \in \mathcal{F}(\Omega)$, $u \geq \varphi_E$ on $\Omega$, and $u = \varphi_E$ on $\Omega \setminus \overline{D}$. In particular,
$$
\mathrm{supp}(dd^c u)^n \subset \Omega \setminus D.
$$

Moreover, by Lemma 4.1 in \cite{ACCH09},
$$
\int_A (dd^c u)^n \leq \int_A (dd^c \varphi_E)^n = 0,
\quad \text{for all pluripolar Borel sets } A \subset \Omega \setminus D.
$$
Hence $u \in \mathcal{F}^a(\Omega)$. By Proposition 3.4 in \cite{CKZ05},
$$
\lim_{t \to +\infty} t^n \,\mathrm{Cap}_\Omega(\{u < -t\}) = 0.
$$
\n
Consequently,
\begin{align*}
&\varliminf_{t \to +\infty} t^n \,
\mathrm{Cap}_\Omega\big(\{\varphi_E < \max(\varphi_E,\psi) - t\}\big) \\
&\quad \leq \varliminf_{t \to +\infty} t^n \,
\mathrm{Cap}_\Omega\big(\{\varphi < \psi - t\} \cap \overline{D}\big)
+ \varlimsup_{t \to +\infty} t^n \,
\mathrm{Cap}_\Omega\big(\{u < -t\} \cap (\Omega \setminus \overline{D})\big) \\
&= 0.
\end{align*}
\n
It thus follows from Lemma \ref{Lemma3} and Lemma 4.1 in \cite{ACCH09} that
\begin{align*}
\int_E (dd^c \varphi)^n
&= \int_{\{\varphi_E = -\infty\}} (dd^c \varphi_E)^n \\
&\leq \int_{\{\max(\varphi_E,\psi) = -\infty\}} (dd^c \max(\varphi_E,\psi))^n \\
&= \int_{\{\max(\varphi_E,\psi) = -\infty\} \cap E} (dd^c \max(\varphi_E,\psi))^n \\
&\quad + \int_{\{\max(\varphi_E,\psi) = -\infty\} \setminus E} (dd^c \max(\varphi_E,\psi))^n \\
&\leq \int_{\{\max(\varphi_E,\psi) = -\infty\} \cap E} (dd^c \psi)^n
+ \int_{\{\max(\varphi_E,\psi) = -\infty\} \setminus E} (dd^c \varphi_E)^n \\
&\leq \int_E (dd^c \psi)^n.
\end{align*}
\end{proof}

\n
{\bf Proof of Theorem \ref{Comparion_Principle_Pluripolar_Set}:}

\n
\begin{proof}
First, assume that $\varphi_i \preceq_{\mathrm{Cap}} \psi_i$ and $\psi_i \preceq_{\mathrm{Cap}} \varphi_i$ on $\Omega$ for all $1 \leq i \leq n$. Then
$$
\sum_{i=1}^n a_i \varphi_i \preceq_{\mathrm{Cap}} \sum_{i=1}^n a_i \psi_i,
\qquad
\sum_{i=1}^n a_i \psi_i \preceq_{\mathrm{Cap}} \sum_{i=1}^n a_i \varphi_i
$$
on $\Omega$ for all $a_1,\dots,a_n \geq 0$. By Lemma \ref{Lemma4},
$$
\int_E \left(dd^c \sum_{i=1}^n a_i \varphi_i\right)^n
=
\int_E \left(dd^c \sum_{i=1}^n a_i \psi_i\right)^n,
$$
for every pluripolar Borel set $E \subset \Omega$ and all $a_i \geq 0$. By polarization, this implies
$$
\int_E dd^c \varphi_1 \wedge \cdots \wedge dd^c \varphi_n
=
\int_E dd^c \psi_1 \wedge \cdots \wedge dd^c \psi_n.
$$

In the general case, we have
$$
\varphi_i \preceq_{\mathrm{Cap}} \max(\varphi_i,\psi_i), \qquad
\max(\varphi_i,\psi_i) \preceq_{\mathrm{Cap}} \varphi_i,
\quad 1 \leq i \leq n.
$$
Applying the first step yields
\begin{align*}
\int_E dd^c \varphi_1 \wedge \cdots \wedge dd^c \varphi_n
&=
\int_E dd^c \max(\varphi_1,\psi_1) \wedge \cdots \wedge dd^c \max(\varphi_n,\psi_n).
\end{align*}
Moreover, Lemma 4.1 in \cite{ACCH09} gives
$$
\int_E dd^c \varphi_1 \wedge \cdots \wedge dd^c \varphi_n
\leq
\int_E dd^c \psi_1 \wedge \cdots \wedge dd^c \psi_n,
$$
for every pluripolar Borel set $E \subset \Omega$.
 
\end{proof}

\n
{\bf Proof of Theorem \ref{Comparion_Principle1}}

\n
\begin{proof}
i) We first prove the theorem under the assumption
$$
\varliminf_{z \to \xi} \big(\varphi(z) - \psi(z)\big) \geq 0,
\quad \forall \xi \in \partial \Omega.
$$
Fix $\epsilon > 0$ sufficiently small, and choose a domain $D$ such that
$$
\{\varphi < \psi - \epsilon\} \Subset D \Subset \Omega.
$$
Then $\varphi \leq \max(\varphi,\psi-\epsilon)$ on $D$, and
$\varphi = \max(\varphi,\psi-\epsilon)$ on $D \setminus \{\varphi < \psi - \epsilon\}$.
By Proposition 3.1 in \cite{KH09},
\begin{align*}
&\frac{1}{n!} \int_D (\max(\varphi,\psi-\epsilon)-\varphi)^n \,
dd^c w_1 \wedge \cdots \wedge dd^c w_n \\
&\quad + \int_D (-w_1)\,(dd^c \max(\varphi,\psi-\epsilon))^n
\leq \int_D (-w_1)\,(dd^c \varphi)^n.
\end{align*}
By Theorem 4.1 in \cite{KH09}, this yields
\begin{align*}
&\frac{1}{n!} \int_D (\max(\varphi,\psi-\epsilon)-\varphi)^n \,
dd^c w_1 \wedge \cdots \wedge dd^c w_n \\
&\quad + \int_{D \cap \{\varphi \leq \psi - \epsilon\}} (-w_1)\,(dd^c \max(\varphi,\psi-\epsilon))^n \\
&\leq \int_{D \cap \{\varphi \leq \psi - \epsilon\}} (-w_1)\,(dd^c \varphi)^n.
\end{align*}
On $\{\varphi > \psi - \epsilon\} \cap \{\varphi \neq -\infty\}$,
$$
(dd^c \max(\varphi,\psi-\epsilon))^n \geq (dd^c \varphi)^n.
$$
Moreover, by Lemma 4.1 in \cite{ACCH09},
$$
(dd^c \max(\varphi,\psi-\epsilon))^n
= (dd^c \max(\varphi,\psi))^n
\quad \text{on } \{\varphi = \psi = -\infty\}.
$$
Combining these estimates, we obtain
\begin{align*}
&\frac{1}{n!} \int_{\{\varphi < \psi - \epsilon\}} (\psi - \epsilon - \varphi)^n \,
dd^c w_1 \wedge \cdots \wedge dd^c w_n \\
&\quad + \int_{\{\varphi < \psi - \epsilon\}} (-w_1)\,(dd^c \psi)^n
+ \int_{D \cap \{\varphi = \psi = -\infty\}} (-w_1)\,(dd^c \max(\varphi,\psi))^n \\
&\leq \int_{\{\varphi < \psi - \epsilon\}} (-w_1)\,(dd^c \varphi)^n
+ \int_{D \cap \{\varphi = \psi = -\infty\}} (-w_1)\,(dd^c \varphi)^n.
\end{align*}
Letting $\epsilon \to 0$ gives
\begin{align*}
&\frac{1}{n!} \int_{\{\varphi < \psi\}} (\psi - \varphi)^n \,
dd^c w_1 \wedge \cdots \wedge dd^c w_n \\
&\quad + \int_{\{\varphi < \psi\}} (-w_1)\,(dd^c \psi)^n
+ \int_{\{\varphi = \psi = -\infty\}} (-w_1)\,(dd^c \max(\varphi,\psi))^n \\
&\leq \int_{\{\varphi < \psi\}} (-w_1)\,(dd^c \varphi)^n
+ \int_{\{\varphi = \psi = -\infty\}} (-w_1)\,(dd^c \varphi)^n.
\end{align*}
\n
(ii) We now consider the case where $\varphi \in \mathcal{N}(H)$ and $\psi \leq H$ on $\Omega$. There exists $u \in \mathcal{N}(\Omega)$ such that
$$
u + H \leq \varphi \leq H.
$$

Let $\{\Omega_j\}_{j \geq 1}$ be a fundamental sequence in $\Omega$, and define
$$
u_j = \sup \{ v \in PSH(\Omega) : v \leq u \ \text{on } \Omega \setminus \overline{\Omega_j} \}.
$$
Then $u_j \searrow u$ on $\Omega$. Since $\psi \leq H$, for every $\epsilon > 0$ we have on $\Omega \setminus \overline{\Omega_j}$
$$
\varphi \geq u + H = u_j + H \geq u_j + \psi - \epsilon.
$$
Applying case (i) to $\varphi$ and $\psi + u_j - \epsilon$, we obtain
\begin{align*}
\frac{1}{n!} &\int_{\{\varphi < \psi + u_j - \epsilon\}} (\psi + u_j - \epsilon - \varphi)^n \,
dd^c w_1 \wedge \cdots \wedge dd^c w_n \\
 + &\int_{\{\varphi < \psi + u_j - \epsilon\}} (-w_1)\,(dd^c(\psi + u_j - \epsilon))^n \\
+ &\int_{\{\varphi = \psi + u_j - \epsilon = -\infty\} \cap \Omega_{j+1}}
(-w_1)\,(dd^c \max(\varphi,\psi + u_j - \epsilon))^n \\
\leq &\int_{\{\varphi < \psi + u_j - \epsilon\}} (-w_1)\,(dd^c \varphi)^n
+ \int_{\{\varphi = \psi + u_j - \epsilon = -\infty\} \cap \Omega_{j+1}}
(-w_1)\,(dd^c \varphi)^n.
\end{align*}

By Lemma 4.1 in \cite{ACCH09}, we have
$$
(dd^c \max(\varphi,\psi + u_j - \epsilon))^n
\geq (dd^c \max(\varphi,\psi))^n
\quad \text{on } \{\max(\varphi,\psi) = -\infty\}.
$$

Hence,
\begin{align*}
\frac{1}{n!} &\int_{\{\varphi < \psi + u_j - \epsilon\}} (\psi + u_j - \epsilon - \varphi)^n \,
dd^c w_1 \wedge \cdots \wedge dd^c w_n \\
 + &\int_{\{\varphi < \psi + u_j - \epsilon\}} (-w_1)\,(dd^c \psi)^n  + \int_{\{\varphi = \psi = -\infty\} \cap \Omega_{j+1}}
(-w_1)\,(dd^c \max(\varphi,\psi))^n \\
\leq &\int_{\{\varphi < \psi + u_j - \epsilon\} \cap \Omega_{j+1}}
(-w_1)\,(dd^c \varphi)^n \\
\leq &\int_{\{\varphi < \psi\} \cap \Omega_{j+1}}
(-w_1)\,(dd^c \varphi)^n
+ \int_{\{\varphi = \psi = -\infty\} \cap \Omega_{j+1}}
(-w_1)\,(dd^c \varphi)^n.
\end{align*}

Letting $j \to +\infty$ and $\epsilon \to 0$, we conclude that
\begin{align*}
\frac{1}{n!} &\int_{\{\varphi < \psi\}} (\psi - \varphi)^n \,
dd^c w_1 \wedge \cdots \wedge dd^c w_n \\
 + &\int_{\{\varphi < \psi\}} (-w_1)\,(dd^c \psi)^n
+ \int_{\{\varphi = \psi = -\infty\}} (-w_1)\,(dd^c \max(\varphi,\psi))^n \\
\leq &\int_{\{\varphi < \psi\}} (-w_1)\,(dd^c \varphi)^n
+ \int_{\{\varphi = \psi = -\infty\}} (-w_1)\,(dd^c \varphi)^n.
\end{align*}
\end{proof}

\n
{\bf Proof of Theorem \ref{Comparion_Principle2}:}

\n
\begin{proof}
Since $\varphi \preceq_{\mathrm{Cap}} \psi$, we have
\[
\varphi \preceq_{\mathrm{Cap}} \max(\varphi,\psi)
\quad \text{and} \quad
\max(\varphi,\psi) \preceq_{\mathrm{Cap}} \varphi.
\]
It thus follows from Theorem \ref{Comparion_Principle_Pluripolar_Set} that
\[
(dd^c \varphi)^n = (dd^c \max(\varphi,\psi))^n
\quad \text{on } \{\varphi = \psi = -\infty\}.
\]

Repeating the arguments in parts (i) and (ii) of Theorem \ref{Comparion_Principle1}, we obtain
\begin{align*}
\frac{1}{n!} \int_{\{\varphi < \psi\}} (\psi - \varphi)^n \,
dd^c w_1 \wedge \cdots \wedge dd^c w_n
+ \int_{\{\varphi < \psi\}} (-w_1)\,(dd^c \psi)^n
\leq \int_{\{\varphi < \psi\}} (-w_1)\,(dd^c \varphi)^n.
\end{align*}
\end{proof}

\n
{\bf Proof of Theorem \ref{Comparion_Principle}:}

\n
\begin{proof}
(i) We first consider the case where
\[
\varliminf_{z \to \xi} (\varphi(z) - \psi(z)) \geq 0,
\quad \forall \xi \in \partial \Omega.
\]
Fix $\epsilon > 0$ sufficiently small, and choose a domain $D$ such that
\[
\{\varphi < \psi - \epsilon\} \Subset D \Subset \Omega.
\]
Then $\varphi \leq \max(\varphi,\psi-\epsilon)$ on $D$, and
$\varphi = \max(\varphi,\psi-\epsilon)$ on $D \setminus \{\varphi < \psi - \epsilon\}$.
By Part (i) of Theorem \ref{Comparion_Principle2}, we have
\begin{align*}
\frac{1}{n!} \int_{\{\varphi < \psi - \epsilon\}} (\psi - \epsilon - \varphi)^n \,
dd^c w_1 \wedge \cdots \wedge dd^c w_n
+ \int_{\{\varphi < \psi - \epsilon\}} (-w_1)\,(dd^c \psi)^n
\\
\leq \int_{\{\varphi < \psi - \epsilon\}} (-w_1)\,(dd^c \varphi)^n.
\end{align*}
Since $(dd^c \varphi)^n \leq (dd^c \psi)^n$ on $\{\varphi < \psi\}$, it follows that
\[
(dd^c \varphi)^n = (dd^c \psi)^n
\quad \text{on } \{\varphi < \psi - \epsilon\}.
\]
Hence the second term cancels, and we obtain
\[
\int_{\{\varphi < \psi - \epsilon\}} (\psi - \epsilon - \varphi)^n \,
dd^c w_1 \wedge \cdots \wedge dd^c w_n = 0,
\]
for all $w_j \in PSH(\Omega) \cap L^\infty(\Omega)$ with $-1 \leq w_j \leq 0$.
This implies $\varphi \geq \psi - \epsilon$. Letting $\epsilon \to 0$, we obtain $\varphi \geq \psi$, as desired.

\medskip

(ii) We now consider the case where $\varphi \in \mathcal{N}(H)$ and $\psi \leq H$ on $\Omega$.
There exists $u \in \mathcal{N}(\Omega)$ such that
\[
u + H \leq \varphi \leq H.
\]

Let $\{\Omega_j\}_{j \geq 1}$ be a fundamental sequence in $\Omega$, and define
\[
u_j = \sup \{ v \in PSH(\Omega) : v \leq u \ \text{on } \Omega \setminus \overline{\Omega_j} \}.
\]
Then $u_j \searrow u$ on $\Omega$. Since $\psi \leq H$, we have on $\Omega \setminus \overline{\Omega_j}$
\[
\varphi \geq u + H = u_j + H \geq u_j + \psi.
\]

On the set $\{\varphi < \psi + u_j\}$ we have
\[
(dd^c \varphi)^n \leq (dd^c \psi)^n \leq (dd^c(\psi + u_j))^n.
\]
Applying case (i) to $\varphi$ and $\psi + u_j$, we obtain
\[
\varphi \geq \psi + u_j.
\]
Letting $j \to \infty$, we conclude that $\varphi \geq \psi$.
\end{proof}

\n
{\bf Proof of Theorem \ref{Comparion_Principle_Equal_Corollary}:}

\n
\begin{proof}
First, we show that
$$\int_E (dd^c \max (\varphi , \psi ) )^n 
= \int_E (dd^c \varphi )^n 
= \int_E (dd^c \psi )^n,$$
for every pluripolar Borel set $E \subset \Omega$. 
Indeed, by Lemma 4.1 in \cite{ACCH09}, we have
$$\int_E (dd^c \max (\varphi , \psi ) )^n \leq \int_E (dd^c \varphi )^n = \int_E (dd^c \psi )^n,$$
for every pluripolar Borel set $E \subset \Omega$. \\
On the other hand, for each $z \in \Omega$, there exists a neighborhood $U$ of $z$ such that either $\varphi \preceq_{\mathrm{Cap}} \psi$ or $\psi \preceq_{\mathrm{Cap}} \varphi$ on $U$.
If $\varphi \preceq_{\mathrm{Cap}} \psi$ on $U$, then by Theorem \ref{Comparion_Principle_Pluripolar_Set},
$$\int_E (dd^c \max (\varphi , \psi ) )^n \geq \int_E (dd^c \varphi )^n,$$
for every pluripolar Borel set $E \subset U$. 
If $\psi \preceq_{\mathrm{Cap}} \varphi$ on $U$, then again by Theorem \ref{Comparion_Principle_Pluripolar_Set},
$$\int_E (dd^c \max (\varphi , \psi ) )^n \geq \int_E (dd^c \psi )^n,$$
for every pluripolar Borel set $E \subset U$. 
Therefore, in either case,
$$\int_E (dd^c \max (\varphi , \psi ) )^n = \int_E (dd^c \varphi )^n = \int_E (dd^c \psi )^n,$$
for every pluripolar Borel set $E \subset U$. 

By the additivity of measures, this equality extends to every pluripolar Borel set $E \subset \Omega$. 

Finally, repeating the arguments from the proof of Theorem \ref{Comparion_Principle}, we conclude that $\varphi = \psi$.

\end{proof}

\n
To prove Theorem \ref{Characterization}, we need the following lemma:

\begin{lemma}\label{Lemma4}
Let $\varphi, \psi \in \mathcal{F}(\Omega)$ be such that $\varphi \leq \psi$ on $\Omega$. Then the following statements are equivalent:

\begin{enumerate}
\item[(i)] There exists a function $u \in \mathcal{F}^a(\Omega)$ such that
$$\varphi \geq \psi + u \quad \text{on } \Omega.$$

\item[(ii)] $\varphi \preceq_{\mathrm{Cap}} \psi$.

\item[(iii)] For every pluripolar Borel set $E \subset \Omega$,
$$\int_E (dd^c \varphi)^n = \int_E (dd^c \psi)^n.$$
\end{enumerate}
\end{lemma}

\begin{proof}
(i) $\Rightarrow$ (ii) By Proposition 3.4 in \cite{CKZ05}, we obtain
$$\varlimsup_{t \to +\infty} t^n \,\mathrm{Cap}_\Omega\bigl(\{\varphi < \psi - t\}\bigr)\leq\lim_{t \to +\infty} t^n \,\mathrm{Cap}_\Omega\bigl(\{u < -t\}\bigr)= 0.$$

\n
(ii) $\Rightarrow$ (iii) The assertion follows from Theorem \ref{Comparion_Principle_Pluripolar_Set} together with the fact that $\psi \preceq_{\mathrm{Cap}} \varphi$ on $\Omega$.

\n
(iii) $\Rightarrow$ (i) By Corollary 4.15 in \cite{ACCH09}, there exist $\varphi_1, \varphi_2 \in \mathcal{F}(\Omega)$ and $\psi_1, \psi_2 \in \mathcal{F}(\Omega)$ such that
$$(dd^c \varphi_1)^n = \mathbf{1}_{\{\varphi = -\infty\}} (dd^c \varphi)^n,\ (dd^c \varphi_2)^n = \mathbf{1}_{\{\varphi > -\infty\}} (dd^c \varphi)^n,$$
$$(dd^c \psi_1)^n = \mathbf{1}_{\{\psi = -\infty\}} (dd^c \psi)^n,\ (dd^c \psi_2)^n = \mathbf{1}_{\{\psi > -\infty\}} (dd^c \psi)^n,$$
and
\[
\varphi_1 \geq \varphi,\quad \varphi_2 \geq \varphi,\quad \varphi \geq \varphi_1 + \varphi_2,
\]
\[
\psi_1 \geq \psi,\quad \psi_2 \geq \psi,\quad \psi \geq \psi_1 + \psi_2.
\]
We also have
\[
\varphi_1 + \varphi_2 \leq \varphi \leq \psi\leq\psi_1.
\]

Moreover, since $\varphi_2 \in \mathcal{F}^a(\Omega)$ and by the proof of (i) $\Rightarrow$ (ii), we obtain $\psi_1 \preceq_{\mathrm{Cap}} \varphi_1$. On the other hand,
\[
(dd^c \varphi_1)^n = (dd^c \psi_1)^n.
\]
By Theorem \ref{Comparion_Principle_Equal_Corollary}, it follows that $\varphi_1 = \psi_1$. Therefore,
\[
\varphi \geq \varphi_1 + \varphi_2
= \psi_1 + \varphi_2
\geq \psi + \varphi_2.
\]
\end{proof}

\n
{\bf Proof of Theorem \ref{Characterization}:}

\begin{proof}

\n
(i) $\Rightarrow$ (ii) For a domain $D \Subset \Omega$, we set
$$u_D = \sup \{ v \in PSH^-(\Omega) : v \leq u \text{ on } D \} \in \mathcal{F}(\Omega).$$
Since $u_D \geq u$ on $\Omega$ and $u \in \mathcal{E}^a(\Omega)$, it follows from Lemma 4.1 in \cite{ACCH09} that $u_D \in \mathcal{F}^a(\Omega)$. Moreover, since $u_D = u$ on $D$, by Proposition 3.4 in \cite{CKZ05} we obtain
$$\varlimsup_{t \to +\infty} t^n \,\mathrm{Cap}_\Omega\bigl(\{\varphi < \psi - t\} \cap D\bigr)\leq\lim_{t \to +\infty} t^n \,\mathrm{Cap}_\Omega\bigl(\{u_D < -t\}\bigr)= 0.$$

\n
(ii) $\Rightarrow$ (i) We set
\[
u(\varphi, \psi) = \sup \{ v \in PSH^-(\Omega) : v \leq \varphi - \psi \} \in PSH^-(\Omega).
\]
We have $$u(\varphi, \psi) \geq \varphi \text{ and }\varphi \geq \psi + u(\varphi, \psi).$$ It suffices to prove $u(\varphi, \psi) \in \mathcal{E}^a(\Omega)$. Indeed, let $W \Subset D \Subset \Omega$. We define
$$\varphi_D = \sup \{ v \in PSH^-(\Omega) : v \leq \varphi \text{ on } D \} \in \mathcal{F}(\Omega),$$
and
$$\varphi_{\Omega \setminus \overline{W}} = \sup \{ v \in PSH^-(\Omega) : v \leq \varphi \text{ on } \Omega \setminus \overline{W} \} \in \mathcal{E}(\Omega).$$
We have
$$\varphi_D \ge \varphi \text{ on } \Omega,\ \varphi_D = \varphi \text{ on } D,\text{ and }
\mathrm{supp}\,(dd^c \varphi_D)^n \subset \overline{D}.$$
Moreover,
$$\varphi_{\Omega \setminus \overline{W}} \ge \varphi \text{ on } \Omega,\
\varphi_{\Omega \setminus \overline{W}} = \varphi \text{ on } \Omega \setminus \overline{W},\text{ and }\mathrm{supp}\,(dd^c \varphi_{\Omega \setminus \overline{W}})^n \subset \Omega \setminus {W}.$$
Since $\varphi_D \preceq_{\mathrm{Cap}} \varphi \preceq_{\mathrm{Cap}} \psi$, we obtain $\varphi_D \preceq_{\mathrm{Cap}} \max(\varphi_D, \psi)$. By the implication (i) $\Rightarrow$ (ii), there exists a function $u_D \in \mathcal{F}^a(\Omega)$ such that
$$\varphi_D \ge \max(\varphi_D, \psi) + u_D \geq \psi + u_D.$$
In particular,
$$\varphi \geq \varphi_D + \varphi_{\Omega \setminus \overline{W}} \geq \psi + u_D + \varphi_{\Omega \setminus \overline{W}}.$$
Hence, $$u(\varphi, \psi) \geq u_D + \varphi_{\Omega \setminus \overline{W}}.$$
By Lemma 4.1 and Lemma 4.4 in \cite{ACCH09}, we obtain
$$\int_E (dd^c u(\varphi, \psi))^n\leq\int_E (dd^c (u_D + \varphi_{\Omega \setminus \overline{W}}))^n = 0,$$
for every pluripolar Borel set $E \subset W$. Letting $W \nearrow \Omega$, we conclude that this equality extends to every pluripolar Borel set $E \subset \Omega$. Therefore, we conclude that $u(\varphi, \psi) \in \mathcal{E}^a(\Omega),$ as desired.
\end{proof}
 
\noindent \textbf{Acknowledgements} This work was completed while the first author was visiting the Department of Mathematics at the National University of Singapore as a Visiting Senior Research Fellow under the Singapore Academies Southeast Asia Fellowship (SASEAF) Programme, within the project “Math, Sobolev Space” (Grant No. E-146-00-0039-01). The first author would like to express his sincere gratitude to the Department of Mathematics at NUS for its warm hospitality and excellent working environment. He also gratefully acknowledges the Singapore National Academy of Science (SNAS) and the National Research Foundation (NRF), Singapore, for their financial support through the SASEAF Programme.
%We would like to thank the anonymous referee for many suggestions that greatly improved the paper.


\begin{thebibliography}{0000}
\bibitem[A07]{A07} P. \AA hag,A Dirichlet problem for the complex Monge–Ampère operator in $\mathcal{F}(f)$, Michigan Math. J. \textbf{55}, 123–138 (2007).
\bibitem[ACCP09]{ACCH09} P. \AA hag, U. Cegrell, R. Czyz, and H.H. Pham, Monge-Amp\`ere measures on pluripolar sets, J. Math. Pures Appl. {\bf 92}, 613--627 (2009).

\bibitem[ACLR25]{ACLR25} P. \AA hag, R. Czyz, H.C. Lu H, and A. Rashkovskii, Geodesic connectivity and rooftop envelopes in the Cegrell classes, Math. Ann. {\bf 391}, 3333--3361 (2025).

\bibitem[BT76]{BT76} E. Bedford and B.A. Taylor, The Dirichlet problem for the complex Monge-Amp\`ere operator, Invent. Math. {\bf 37}, 1--44 (1976).

\bibitem[BT82]{BT82} E. Bedford and B.A. Taylor, A new capacity for plurisubharmonic functions, Acta Math. {\bf 149}, 1--40 (1982).

\bibitem[BT87]{BT87} E. Bedford and B. A. Taylor, Fine topology, Silov boundary, and $(dd^c)^n$, J. Funct. Anal. {\bf 72}, 225--251 (1987).
\bibitem[Be06]{Be06} S. Benelkourchi, A note on the approximation of plurisubharmonic functions, C. R.
Math. Acad. Sci. Paris \textbf{342}, 647--650 (2006).
\bibitem[B06]{B06} Z. B\l ocki, The domain of definition of the complex Monge–Amp\`ere operator, Am. J. Math. \textbf{128}, 519--530 (2006).
\bibitem[C98]{Ce98} U. Cegrell, Pluricomplex energy, Acta Math. {\bf 180}, 187--217 (1998).

\bibitem[C04]{Ce04} U. Cegrell, The general definition of the complex Monge-Amp\`{e}re operator, Ann. Inst. Fourier (Grenoble) {\bf 54}, 159--179 (2004).
\bibitem[C08]{Ce08} U. Cegrell, A general Dirichlet problem for the complex Monge–Amp\`{e}re operator, Ann. Polon. Math. \textbf{94}, 131--147 (2008).
\bibitem[CKZ05]{CKZ05} U. Cegrell, S. Kolodziej and A. Zeriahi, Subextension of plurisubharmonic functions with weak singularities, Math. Z. {\bf 250}, 7--22 (2005).

\bibitem[D93]{De93} J.-P. Demailly, Monge-Amp\`ere operators, Lelong numbers and intersection theory, in Complex Analysis and Geometry, Univ. Ser. Math., pp. 115–193. Plenum, New York, 1993. 
\bibitem[DD20]{DD20} H-.S. Do and T.D. Do, Some remarks on the Cegrell's class $\mathcal {F} $, Ann. Polon. Math. \textbf{125}, 13--24 (2020).
\bibitem[NP09]{KH09} V.K. Nguyen and H.H. Pham, A comparision principle for the complex Monge-Amp\`ere operator in Cegrell’s classes and applications, Trans. Am. Math. Soc. {\bf 361}, 5539--5554 (2009).

\end{thebibliography}
\end{document}